\documentclass[12pt]{amsart}
\usepackage{amsmath,amssymb,amsthm,enumerate,marginnote}
\newtheorem{theorem}{Theorem}

\newtheorem{proposition}[theorem]{Proposition}
\newtheorem{lemma}[theorem]{Lemma}
\newtheorem{corollary}[theorem]{Corollary}

\newtheorem{problem}[theorem]{Problem}
\newtheorem{property}{Property}
\theoremstyle{definition}

\theoremstyle{remark}
\newtheorem*{remark}{Remark}
\newtheorem*{notation}{Notation}
\newcommand{\Addresses}{{
  \bigskip
  \footnotesize

{Department of Mathematics, The University of Chicago, 5734 S. University Avenue, Chicago, IL 60637, USA. Email:
iqra@uchicago.edu}

\medskip

{Department of Mathematics, The University of Chicago, 5734 S. University Avenue, Chicago, IL 60637, USA. Email:
csornyei@math.uchicago.edu}

}}

\def\R{{\mathbb R}}
\def\RR{{\mathbb R}}
\def\NN{{\mathbb N}}
\def\al{{\alpha}}
\def\eps{{\varepsilon}}
\def\om{{\omega}}
\def\dist{{\mathrm{dist}}}

\def\cl{{\mathrm{cl}}}
\def\emp{{\emptyset}}

\def\II{{{\mathcal I}}}
\def\HH{{{\mathcal H}}}

\title[Metric embeddings into $\R^d$ without shrinking]{On the modulus of continuity of functions whose image has positive measure, and metric embeddings into $\R^d$ without shrinking}
\author{Iqra Altaf and Marianna Cs\"ornyei}
\date{}
\begin{document}
\begin{abstract}
A generalization of the classical Sard theorem in the plane is the following. Let $f$ be a function defined on a subset $A\subset{\mathbb R}^2$. If $f$ has modulus of continuity $\omega(r)\lesssim r^2$, then $f(A)\subset{\mathbb R}$ has Lebesgue measure zero. Choquet claimed in \cite{Choquet} that this was a full characterization, i.e. for every $\omega$ for which $\omega(r)/r^2$ converges to $\infty$ as $r\to 0$, there is a counterexample. We disprove this by showing that the correct characterization, in $\R^d$, is $\int_{0}^{1} \omega(r)^{-1/d}=\infty$. For the precise statement see Theorem \ref{our}.

We obtain this as a special case of a more general result. We study which spaces $(X,\rho)$ can be embedded into ${\mathbb R}^d$ without decreasing any of the distances in $X$. That is, we ask the question whether there is an $f:\,X\to {\mathbb R}^d$ such that $\|f(x)-f(y)\|\ge \rho(x,y)$ for every $x,y\in X$. We study this problem for some very general distance functions $\rho$ (we do not even assume that it is a metric space, in particular, we do not assume that $\rho$ satisfies the triangle inequality), and find quantitative necessary and sufficient conditions under which such a mapping exists.

We will obtain the characterization mentioned above as a special case of our metric embedding results, by choosing $X$ to be an interval in $\R$, and defining $\rho$ by putting $\rho(x,y)=r$ if $\|x-y\|=\omega(r)$. 
\end{abstract}
\maketitle

\tableofcontents
\section{Introduction}

\subsection{Sard: a brief history}
Let $f\colon\mathbb{R}^{n} \rightarrow \mathbb{R}^m$ be a differentiable function. The critical points of $f$ are the points at which the Jacobian matrix of $f$ has rank less than $m$, and the critical values are the values of $f$ at the critical points. The classical Sard theorem (also known as the Sard lemma or Morse-Sard theorem) says that if a function $f$ is $C^k$ differentiable with $k \geq \max(n-m+1,1)$ then the set of the critical values of $f$ in $\mathbb{R}^m$ has Lebesgue measure $0$. The case $m=1$ was shown by Morse \cite{Morse}, and the general case by Sard \cite{Sard0}. Later Sard in \cite{Sard1} generalized his theorem to show that if $k\ge\max(n-m+1,1)$, $f$ is $C^k$ and $X_d$ is the set of points at which the rank of the Jacobian matrix is strictly less than $d$, then the $d$-dimensional Hausdorff measure of $f(X_d)$ is zero. 

Bates \cite{Bates} lowered the regularity assumption in the original Sard theorem by showing that if $n>m$ and $f\in C^{n-m,1}(\R^n,\R^m)$ then the critical values of $f$ has Lebesgue measure zero. It does not hold for $f\in C^{n-m,\alpha}(\R^n,\R^m)$ with $\alpha<1$. \par
Dubovitskii and Federer \cite{Dubo}, \cite{Federer} independently gave another generalization of Sard theorem as the following result. Let $n,m,k,d \in \mathbb{N}$, $d \le\min\{n,m\}-1$,  $v=d+\frac{n-d}{k}$ and $f\in C^k(\mathbb{R}^n,\mathbb{R}^{m})$. Then the $v$-dimensional Hausdorff measure of $f(X_d)$ is zero. (Putting $d=m-1$ and $k=n-m+1$ gives the classical Sard theorem.) Recently in \cite{Ferone} this has been generalized to $C^{k, \alpha}$ functions by Ferone, Korobkov and Roviello. They showed that if $n,\, m,\,k \in \mathbb{N}$, $0\leq \alpha \leq 1$, $d < m$, $f \in C^{k, \alpha}(\mathbb{R}^n, \mathbb{R}^m)$, $q \in (d, \infty)$ and $v=n-d-(k+\alpha)(q-d)$, then 
    $$\mathcal{H}^{v}(X_d \cap f^{-1}(y))=0 \text{ for } \mathcal{H}^{q}\text{-a.e. } y \in \mathbb{R}^m. $$
In this theorem with $\alpha=0$ and $q=d+\frac{n-d}{k}$ we get the above-mentioned result of Dubovitskii and Federer. And with $\alpha=1$, $k=n-m$, $q=m$ and $d=m-1$ we obtain Bates theorem for $C^{n-m,1}(\mathbb{R}^n, \mathbb{R}^m)$ functions.

Sard theorem has also been generalized in the context of Sobolov spaces.  
In \cite{Pasc}, Pascale showed that if $n<p$ and $k=n-m+1$ then Sard theorem holds for $W_{\text{loc}}^{k,p}(\mathbb{R}^n, \mathbb{R}^m)$. For many other generalizations of Sard theorem see \cite{Alberti}, \cite{Bourgain}, \cite{Bourgain1}, \cite{Ferone}, \cite{Figalli}, \cite{Norton}, \cite{Pavlica}, \cite{Putten1} and \cite{Smale}.

\subsection{Our generalization, modulus of continuity}

Let $f\in C^{1,1}(\R^2,\R)$, and let $A$ be the critical set of $f$, that is, the set of points at which the gradient of $f$ is zero. Then the restriction of $f$ onto $A$ has modulus of continuity $\om(r)\lesssim r^2$. It is an easy exercise to show that, for every $d\ge 2$ and for every $A\subset\R^d$, if a function $f\colon A\to \RR$ has modulus of continuity $\om(r)\lesssim r^d$ then $f(A)$ is null.

One of our aims in this paper is to study which functions $r^d$ can be replaced by so that the conclusion of Sard theorem remains true. We ask the following question:

\begin{problem}\label{qsard}
  For which $\omega$ can one find a set $A\subset\RR^d$ and a function $f\colon A\to \RR$ with modulus of continuity $\om$ such that $f(A)$ has positive measure?
\end{problem} 

Based on the discussion above, it is natural to consider only modulus of continuities $\om(r)$ for which $\om(r)/r^d\to\infty$ as $r\to 0$. Choquet showed (see page 52 in \cite{Choquet}) the following statement:
If $\om(r)/r^2$ converges to infinity as $r\to 0$, then there exist a set $A\subset \RR^2$ and a function $f\colon A\to\RR$ with modulus of continuity $\om$ s.t. $f(A)$ has positive measure.
The second author announced in several conferences (see e.g. \cite{Marianna}) that Choquet's argument was incorrect and, in fact, there was a counterexample for his statement: for every $A\subset\R^2$ and $f\colon A\to \R$ with modulus of continuity $\om(r)=r^2|\log r|$ the measure of $f(A)$ is zero.

In this paper we give the following characterization:

\begin{theorem}\label{our}
 Let $d\geq 2$. Assuming that $\om(r)$ is strictly increasing, convex, $\omega(0)=0$ and $\om(r)/r^d$ is monotone decreasing, the following are equivalent:
  \begin{enumerate}[(i)]
  \item For every $A\subset\RR^d$ and for every $f\colon A\to\RR$ with modulus of continuity $\om$, the measure of $f(A)$ is zero.
  \item $\int_0^{1}\om(r)^{-1/d}\,dr =\infty$.
  \end{enumerate}
\end{theorem}

\bigskip
We will obtain Theorem \ref{our} as a special case of some general metric embedding results that we discuss in the next sections. The precise statements of these embedding results are quite technical (see Section \ref{mainresults}).

In Section \ref{bgrowth} we discuss the metric properties that we need to state and prove our general results. We will focus in Section \ref{bgrowth} on two interesting special cases: the first one is the Sard example (i.e. Theorem \ref{our}). The other special case, which we call 'Cantor example' can be thought of as a sort of discretized version of the continuous Sard problem.

The Cantor case is easier than the Sard case. For instance, in the Cantor case many of the general constants used in our proof will be 1 (or 0); and while in the Sard case we need to estimate how certain sets overlap, in the Cantor case they will be disjoint. We hope that seeing how our proof works in the simpler Cantor example will also help the reader to understand some of the more delicate details of the Sard (and of the general) proof.

Section \ref{mainproof} and Section \ref{construction} contain the proof of our main result. In Section \ref{mainproof} we prove the non-trivial direction; in Section \ref{construction} we present a very easy construction showing the sharpness of our results. This construction has a very similar flavour to the one that Choquet was using in his paper -- we show that, under the correct assumptions, it indeed gives a function whose image has positive measure.

\subsection{Embeddings into $\R^d$.}
We study the following question. Suppose that we are given a set $X$, and a $\rho\colon X\times X\to \RR^+$. (We do not assume that $\rho$ is a metric, in particular, it does not need to satisfy the triangle inequality.)

\begin{problem}\label{qembed}
  Does there exist a mapping $F\colon X\to \RR^d$ s.t. $F(X)$ is bounded, and
\begin{equation}\label{F}
\rho(x,y)\le|F(x)-F(y)|
\end{equation}
for every $x,y\in X?$
\end{problem}

In Problem \ref{qsard} without loss of generality we can assume that the set $A$ is bounded. We can also assume that $f(A)$ is an interval, since every set of positive measure in $\R$ can be mapped onto an interval by a contraction. Therefore Problem \ref{qsard} becomes a special case of Problem \ref{qembed} by choosing $X$ to be an interval, and \begin{equation}\label{rhomega}\rho(x,y):=r\quad\mathrm{if}\quad |x-y|=\om(r).\end{equation}

Another interesting special case to keep in mind is the following. Suppose that $X$ is a Cantor set, and we consider a ‘distance function' on the Cantor set given by a sequence $r_n\searrow 0$ as follows: $\rho(x,y)=r_n$ if $x$ and $y$ were separated in the $n^{th}$ step of the standard Cantor set construction. For which sequence $\{r_n\}$ can this Cantor set be embedded into $\R^d$ without decreasing any of the distances, i.e. s.t. \eqref{F} holds?

Another way of thinking about the Cantor set example is the following. As usual, we identify the $2^n$ subsets in the Cantor set construction by the $2^n$ sequences of 0's and 1's of length $n$, which we denote by $\II_n$. We are looking for mappings s.t. $|F(x)-F(y)|$ is at least $r_n$ if $x$ and $y$ don't belong to the same $I\in\II_n$.

More generally, let $\mathcal{I}$ denote the set of all finite sequences of $0$'s and $1$'s. For every space $X$, we can:

\begin{enumerate}
\item Choose a ‘distance function' $\rho(I,J)$ for every $I,J\in\II$.
\item For every $I\in\II$, choose a set $X_I\subset X$.
\item Ask the question whether there is an $F\colon X\to\R^n$ s.t. $F(X)$ is bounded, and 
  $$\dist(F(X_I),F(X_J))\ge \rho(I,J)$$ for every $I,J\in\II$.
\end{enumerate}
In the Sard problem, when $X$ is an interval, we can assume this interval contains $(0,1)$, and then choose $X_I$ to be the open dyadic subinterval of $(0,1)$ encoded by the sequence $I$.

Also in the general case,  it is convenient to keep the notation coming form the dyadic intervals idea, i.e. we denote $J\subset I$ if the sequence $J$ is a continuation of the sequence $I$, and $I\cap J=\emp$ if they are not continuations of each other (for each $I,J$, either $I\subset J$ or $J\subset I$ or $I\cap J=\emp$). Using this notation, we also assume:

\begin{enumerate}
\item[4.] $X_I\subset X_J$ if $I\subset J$.
\item[5.] $X_I\cap X_J=\emp$ if $I\cap J=\emp$. (We do not assume $\rho(X_I,X_J)>0$.)
\end{enumerate}

We can now forget about the underlying space $X$ and study $\rho$ by itself.
We assume that we are given a ‘distance function' $\rho\colon\II\times\II\to\R$ s.t. for every $I,J\in\II$:
\begin{itemize}
\item $\rho(I,J)=0$ if $I\subset J$ or $J\subset I$, and $\rho(I,J)\ge 0$ for every $I,J$;
\item $\rho(I,J)=\rho(J,I)$;
\item $\rho(I,K)\le\rho(J,K)$ if $J\subset I$. 
\end{itemize}

Our main question is the following:
\begin{problem}\label{main}
  Can one find for every $I\in\II$ a bounded set $A_I\subset \RR^d$ s.t.
  \begin{enumerate}[(i)]
  \item $A_I\subset A_J$ if $I\subset J$;
  \item $A_I\cap A_J=\emp$ if $I\cap J=\emp$;
  \item $\dist(A_I,A_J)\ge \rho(I,J)$ if $I\cap J=\emp$?
  \end{enumerate}
\end{problem}
\begin{remark}
  If for every $x,y\in X$, $$\rho(x,y)=\sup\{\rho(I,J)\colon x\in X_I,y\in X_J\},$$ then of course by choosing for each $x\in X$ a point $$F(x)\in\bigcap_{x\in X_I}\rm{cl}\,(A_I),$$
  \eqref{F} will hold.
\end{remark}

We cannot answer this question in its full generality, but we will try to answer it under some 'bounded growth' assumptions, which we will discuss in the next section. These assumptions will be satisfied in the Sard problem (i.e. when $\rho$ is defined by \eqref{rhomega}, with a modulus of continuity whose properties satisfy the assumptions in Theorem \ref{our}). And our bounded growth assumptions will also be automatically satisfied in the Cantor set example discussed above, when $\rho$ is given by a decreasing sequence $\{r_n\}$. \par
 
\bigskip
\begin{notation} For any set $A \subset \mathbb{R}^d$, we denote by $|A|$ the Lebesgue measure and by ${\mathcal {H}}^s(A)$ the $s$-dimensional Hausdorff measure of $A$, respectively. We denote by $B(A,r)$ the open $r$-neighborhood of $A$. We use the standard notation $a \lesssim b$ for the existence of an absolute contant $c$ such that $a \le cb$, and $a\sim b$ stands for the existence of an absolute contant $c$ such that both $a \le cb$ and $b \le ca$ hold. 
\end{notation}

\section{Bounded growth conditions}\label{bgrowth}

\subsection{Properties of $\om$}
Recall that we assumed in Theorem \ref{our} that $\omega\colon [0,\infty)\to [0,\infty)$ is a continuous, strictly increasing, convex function such that $\om(0)=0$ and $\om(r)/r^d$ is decreasing. In this section we state some important corollaries that we will use in the coming sections to show that our 'bounded growth assumptions' are automatically satisfied in the Sard example. The first part of Lemma \ref{Omm} says that $\omega$ is doubling. The heart of the matter in our proof in the Sard case will be part (ii) of Lemma \ref{Omm} which we will refer to as the 'Lipschitz' condition. We will explain its geometric meaning in Section \ref{dij}.

\begin{lemma}\label{Omm} Let $\omega\colon[0,\infty)\to [0,\infty)$ be a continuous, strictly increasing, convex function such that $\om(0)=0$ and $\om(r)/r^d$ is decreasing. Then:
\begin{enumerate}
\item[(i)] $\om(2r)\lesssim \om(r)$ for every $r>0$.
\item[(ii)] $\omega(r+\alpha R)-\omega(r) \lesssim \alpha \omega(R)$ for every $R\geq r>0$ and $\alpha \in [0,1]$.
\end{enumerate}
\end{lemma}

\begin{proof}
Using the notation $\om(r)=r^du(r)$, the function $u(r)$ is decreasing. Therefore 
  $$\om(2r)= 2^dr^d u(2r) \leq 2^d r^du(r)=2^d\om(r),$$ so indeed (i) holds. For every $R\geq r>0$, we have 
 \begin{align}\label{Rr}
      \om(R)-\om(r)&=R^du(R)-r^du(r)=(R^d-r^d)u(R)+r^d(u(R)-u(r))\nonumber\\
                   &\le(R^d-r^d)u(R)\le d(R-r)R^{d-1}u(R)=d(R-r)\om(R)/R.
 \end{align}
 And by convexity, we have 
 \begin{equation}\label{convex}
   \omega(r+\alpha R)-\omega(r) \leq \omega(R+\alpha R)-\omega(R).
   \end{equation}
By replacing $R$ by $R+\al R$ and replacing $r$ by $R$ in \eqref{Rr}, we get
\begin{equation}\label{Ralpha}
    \omega(R+\alpha R)-\omega(R) \leq d\alpha R \omega(R+\alpha R)/(R+\alpha R)\lesssim \alpha \omega(2R).
\end{equation}
Combining \eqref{convex}, \eqref{Ralpha} and (i) gives (ii).
\end{proof}

\subsection{The sequence $\{r_n\}$ and the sets $C_I$}\label{seci}

Now consider the general case. For each $J\subset I$ denote 

$$\rho_{J,I}:=\inf\{\rho(J,K)\colon K\cap I=\emp\}$$

and $$\rho_I:=\sup_{J\subset I}\rho_{J,I}.$$

In the Cantor set example, if $I\in\II_n$ then everything in $I$ has distance (at least) $r_n$ from the complement of $I$, so in this case $\rho_{J,I}=\rho_I=r_n$ for every $J\subset I$. In the Sard example (identifying the sequence $I$ with the corresponding dyadic interval $I\subset(0,1)$), $\rho_{J,I}=r$ is defined by $\om(r)=\dist(J,I^c)$, and therefore $\rho_I=r$ is given by $\om(r)=1/2^{n+1}$. In both examples, $\rho_I$ depends only on the length of the sequence $I$. We will also assume in the general case that $r_n:=\rho_I$ depends only on $n$, for $I\in\II_n$.

Now we are ready to state our first bounded growth condition:
  There are some absolute constants s.t. for each $I\in\II_n$ and for each $m\ge n$, and for all except at most constant many $J\in \II_m$ outside $I$,
    $$\rho(I,J)\gtrsim r_m.$$

 Heuristically, this captures the fact that our dyadic intervals have a linear structure, they are not, say, dyadic squares. More precisely, their boundary is '0 dimensional': in every `annulus' of width $\sim r_m$ there are only constant many intervals of size $\sim r_m$. (In the Sard example, this means we are looking at the real valued, not the vector valued case.) 

 This 0 dimensional boundary condition is obviously true when our intervals are indeed intervals, and $\omega$ is doubling (see property (i) in Lemma \ref{Omm}), i.e. our bounded growth assumption holds trivially in the Sard case. And it also obviously holds in the Cantor set example, with no exceptional $J$'s, because of our special distance function $\rho$ we chose in that case.

 Why are we happy about this? It has a geometric meaning. Remember that we are looking for embeddings into $\R^d$, satisfying the conditions in Problem \ref{main}. If such a mapping exists, denote 
 \begin{equation} \label{Eq1}
     C_I:=\bigcup_{J\subset I}B(A_J,c\rho_{I,J})\subset\R^d
 \end{equation} with some sufficiently small constant $c$, given by the bounded growth assumption. By the definition of $\rho_{J,I}$, if $c\le 1/2$ then these sets are pairwise disjoint. But thanks to the bounded growth assumption, we can say more: since for every interval $I\in\II_n$, for every  $m\ge n$,  and for all but constantly many $I'\in\II_m$, and for any $J\subset I$, $J'\subset I'$ there holds $\rho(J,J')\gtrsim\max(\rho_{J,I},\rho_{J',I'},r_m)$, therefore for these intervals and for $c$ sufficiently small enough $\rho(J,J')\geq c(\rho_{J,I}+\rho_{J',I'}+2r_m)$. Hence:
 \begin{property}\label{CI} For each $I\in\II_n$ and for each $m\ge n$, there are at most constant many $I'\in \II_m$ s.t. $B(C_I,cr_m)\cap B(C_{I'},cr_m)\neq\emp$.
 \end{property}

 In our main result we will characterize for which $\rho$ can such sets $C_I$ with small overlaps of their neighborhoods exist in $\R^d$. We will not actually use how the sets $C_I$ were defined, and we will forget what $\rho$ was, we will only assume that the sets $C_I$ have some nice properties implied by $\rho$. One of the key assumptions we will use is that Property \ref{CI} holds.

 Another key property we will use (which follows immediately from the definition of our $C_I$ above) is that the sets $C_I$ are not too small:

 \begin{property}\label{rn}
   For each $I\in\II_n$, $C_I$ contains a ball of radius $\sim r_n$.
 \end{property}
 \subsection{The sets $E_{I,j}$ and their annuli $D_{I,j}$}\label{dij}

Consider first the Sard problem. Recall our ‘Lipschitz' condition (ii) in Lemma \ref{Omm}: 
\begin{equation}\label{Lipschitz}
  \om(r+\al R)-\om(r)\lesssim \al \om(R)\end{equation}
for every $r\le R$ and $\al\in [0,1]$.


This again has a very nice geometric corollary, which we now explain. In the Cantor set example we won't have an analogue of the Lipschitz condition \eqref{Lipschitz}, but the same geometric corollary will hold also in that case. As in the previous section, we will later care only about this corollary and not about how our sets were defined.

In the Sard case it is natural to consider every interval $I\subset[0,1]$, not only dyadic intervals. For arbitrary intervals $J\subset I$ we can define $\rho_{J,I}$ to be the $\rho$ distance of $J$ from the complement of $I$, and then define $C_I$ by \eqref{Eq1}. In \eqref{Eq1}, $I$ is now not necessarily a dyadic interval, but the subintervals $J$ are dyadic (we haven't defined the sets $A_J$ for non-dyadic intervals).

By exactly the same argument as in the previous section, if $I\cap I'=\emp$ then $C_I\cap C_I'=\emp$, and moreover, using the same argument we will show that the following generalization of Property \ref{CI} hold: if $I$ is arbitrary, $J\in\II_m$, $|I|\ge |J|$, and the $\rho$-distance between $I$ and $I'$ is at least $r_m$, then
\begin{equation}\label{uj}
B(C_I,cr_m)\cap B(C_J,cr_m)=\emp.\end{equation}
In order to show this  we need to check that for every $J\subset I$ and $J'\subset I'$, the distance of $A_J$, $A_J'$ is at least $c\rho_{J,I}+c\rho_{J',I'}+2cr_m$, which we indeed have since $\rho(J,J')\ge\max(\rho_{J,I},\rho_{J',I'},r_m)$.

So far nothing happened, everything is exactly the same as in the previous section. But $\eqref{Lipschitz}$ allows us to say something more delicate: for every $J\subset I\subset I'$, the Lipschitz condition $\eqref{Lipschitz}$ tells us that there is an absolute constant $L$ s.t. $\rho_{J,I'}-\rho_{J,I}\ge \alpha R$ if $\om(\rho_{J,I'})-\om(\rho_{J,I})\ge L\alpha \om(R)$ and $\rho(J,I) \leq R$. In particular, if the Euclidean distance of $I$ from the complement of $I'$ is $\ge L\al\om(R)$ and $\rho(J,I) \leq R$ then $\rho_{J,I'}-\rho_{J,I}\ge \alpha R$, for all $J \subset I$.

For any $d>0$, let $I_d$ denote the Euclidean $d$-neighbourhood of $I$. Then with $I'=I_d$, $d=L\al\om(R)$, the above discussion tells us that $\rho_{J,I}+\al R\le \rho_{J,I'}$, i.e.,

\begin{equation}\label{La}B(C_I,\alpha R)\subset C_{I_{L\al \om(R)}}\end{equation} for every $\alpha\le 1$, provided that $\rho_{J,I}\le R$ for every $J\subset I$.

Now let $I\in\II_n$, $m\ge n$, and for each $j\ge 0$ denote
$$E_{I,j}:=C_{I_{j2^{-m}}},\quad D_{I,j}=B(E_{I,j},cr_m)\setminus E_{I,j}.$$
(Slightly abusing the notation, we put $I_0=I$, i.e. $E_{I,0}=C_I$. Also, the sets $E_{I,j}$, $D_{I,j}$ depend not only on $I$ and $j$ but also on $m$. We will always fix an $m$ before using them therefore we hope that the ambiguity of our notation will not cause any confusion.)  For future reference, note that

\begin{property}\label{3} $C_I \subset E_{I,j}$ for every $I\in\II_n$ and for every $j$.
\end{property}

From \eqref{La}, with $R\sim r_n$, $\om(R)\sim 2^{-n}$, and $\al R= cr_m$, i.e. $\al=cr_m/R$ we get 

\begin{equation}\label{ujjj}
  D_{I,j}\subset B(E_{I,j},cr_m)=B(C_{I_{j2^{-m}}},cr_m)\subset C_{I_{j2^{-m}+L\al\om(R)}}\subset C_{I_{j'2^{-m}}},
\end{equation}if $$j'-j\ge L\al\om(R)2^m\sim 2^{m-n}r_m/r_n.$$ In order for this calculation to be valid we need $\al\le 1$ (which we have if $c$ is small enough), but we also need that for every $J\subset I_{j2^{-m}}$ there holds $\rho_{J,I_{j2^{-m}}}\le R$. So (using the doubling property of $\omega$) it holds for every $0\le j< k_{m,n}$ with some $$k_{n,m}\sim 2^{m-n}.$$

Note that since $\omega$ is convex, therefore $\om(r_m)/r_m\le \om(r_n)/r_n$, and since $\om(r_m)\sim 2^{-m}$, $\om(r_n)\sim 2^{-n}$, therefore
$$k_{n,m}r_m/r_n\sim 2^{m-n}r_m/r_n\gtrsim 1.$$
We will show the following two properties: \\

\begin{property}\label{4}
  There is an absolute constant s.t. for each $J \in \mathcal{I}_{m}$ there are only constant many pairs $(I,j)$ with $I\in\II_n$, $j< k_{n,m}$ s.t. $D_{I,j}\cap C_{J}\neq \emp$.
\end{property}

\begin{property}\label{5} For every $m>n$, $\sum_{I\in\II_n,j<k_{n,m}}\chi_{D_{I,j}(x)}\lesssim k_{n,m}r_m/r_n$ for every $x$.
\end{property}

Property \ref{5} follows from the observations that:
\begin{itemize}
\item[(i)] For $j'-j\gtrsim 2^{m-n}r_m/r_n$, we have $C_{I_{j2^{-m}}}\cap D_{I,j}=\emp$ and $D_{I,j}\subset C_{{I}_{j'2^{-m}}}$. Therefore $D_{I,j}\cap D_{I,j'}=\emp$ and for each $I$ and $x \in \mathbb{R}^d$, $\sum_{j <k_{n,m}}\chi_{D_{I,j}(x)}\lesssim k_{n,m}r_m/r_n$.
\item[(ii)] For every $I,j,I',j'$, we have $D_{I,j}\subset C_{I_{2^{-n}}}$ and $D_{I',j'}\subset C_{{I'}_{2^{-n}}}$, therefore if $\dist(I,I')> 2/2^{-n}$, then $D_{I,j}\cap D_{I',j'}=\emp$. And for each fixed $I$ there are only constant many $I'$ with $\dist(I,I')\le 2/2^{-n}$.
\end{itemize}

In order to prove Property \ref{4}, we use \eqref{uj}. If $D_{I,j}\cap C_J\neq\emp$, then the distance of $E_{I,j}=C_{I_{j2^{-m}}}$ and $J$ is at most $2^{-m}$. Clearly, there are only constantly many such $I$, and for each $I$ there is only one $j$ for which $D_{I,j}\cap C_J\neq\emp$. 

\bigskip
Now consider the Cantor case. In that example we don't have any non-dyadic intervals, and we don't have any analogue of the Lipschitz condition $\eqref{Lipschitz}$. We cannot repeat the same construction, we can instead find sets $E_{I,j}$ and their annuli $D_{I,j}$ satisfying Properties \ref{3}-\ref{5} using a much more trivial construcion. We will show that in the Cantor case we can choose $k_{n,m}$ to be as large as we like.

Fix $n\le m$ and a $k_{n,m}\gtrsim r_n/r_m$. Recall that in the Cantor case we have $C_I=\bigcup_{J\subset I}B(A_J,cr_n)$ for $I\in\II_n$, with some sufficiently small constant $c$. We assume $c<1/6$, and for every $j<k_{n,m}$ define $$E_{I,j}:=B(C_I,cjr_{n}/k_{n,m}),\quad D_{I,j}=B(E_{I,j}, cr_m)\setminus E_{I,j}.$$

Since $D_{I,j}\subset B(E_{I,j}, cr_m)\subset B(C_I,2cr_n)\subset B(A_I,3cr_n)\subset B(A_I,r_n/2)$, therefore $D_{I,j}$ cannot intersect any $C_J$, so Property 4 holds trivially. Moreover each $D_{I,j}$ intersects $\lesssim k_{n,m}r_{m}/r_n$ many other annuli, so indeed Property 5 holds.

\begin{remark}
Heuristically, the fact that we can choose $k_{n,m}$ to be arbitrarily large says that the Cantor set exhibits properties that a 0-dimensional set has. The fact that for the interval we have $k_{n,m}\sim 2^{m-n}$ corresponds to the fact that the interval is 1-dimensional. In fact, our proof will only use that $k_{n,m}$ increases faster than $2^{(m-n)/d}$; heuristically this says that (that part of) our proof works as long as our set is less than $d$-dimensional.
\end{remark}

\begin{remark} The statement of Theorem \ref{our} is obviously false with $d=1$. We use $d>1$ when we say that $2^{m-n}$ increases faster than $2^{(m-n)/d}$.  
\end{remark}

\subsection{Boundary}\label{section5}

In this section we will describe two more properties, which are very easy to check both in the Sard and in the Cantor case. They will again state a sort of ‘0-dimensional boundary' type condition.

We will now consider three scales, $m>n>k$. We fix a large enough absolute constant $q$, and assume that $n\ge k+q$, $m\ge n+q$. We fix $k$, and also fix an interval $I\in\II_k$. Our claim is that for every $n\ge k+q$ there is an $\II_n'\subset\{I'\in \II_n\colon I'\subset I\}$, s.t. $\II_n'\neq \emp$, and the following hold.

\begin{remark} We want to think of this as intervals $I'$ in $I$ not too close to the complement of $I$. In the Cantor set example, nothing in $I$ is close to the complement of $I$, and therefore $\II_n'$ will be the set of all intervals $I'\in\II_n$ contained in $I$. And we can take $q=1$. In the Sard example, we will need to throw away intervals close to each of the endpoints. The statement will be true with $q=3$.
\end{remark}

The properties we will show: for every $n\ge k+q$ and for every $m\ge n+q$, the sets $D_{J,j}$ in the previous section satisfy the following.

\begin{property}\label{6} $D_{I',j}\subset C_I$ for every $I'\in \II_{n}'$, $j<k_{n,m}$.
\end{property}

\begin{property}\label{7} If $D_{I',j}\cap C_{I''}\neq\emp$ for some $I'\in \II_{n}'$, $j<k_{n,m}$ and $I''\in\II_m$, then $I''\in\II_m'$.
\end{property}

In the Cantor case this is obvious. In the Sard case, we divide $I$ into $2^{n-k}$ subintervals and throw away the first two and the last two. With $q\ge 3$ some intervals will be left, i.e. $\II_n'\neq \emp$.

We have already seen in the last section that $D_{I',j}\subset C_{I_{2^{-n}}}$, therefore $D_{I',j}\subset C_I$ for all except possibly the first and the last $I'$. And we have also seen that if $D_{I',j}\cap C_{I''}\neq \emp$, then $\dist(I',I'')\le 2^{-n}$. So $I''$ is not contained in the first or last $I'$, and then of course it cannot be the one of the first or last two $I''$.
  
\section{Main result}\label{mainresults}

\subsection{Main geometric result in $\R^d$} 
For the convenience of the reader, first we re-list Properties 1-7, with $cr_n$ replaced by $r_n$ everywhere. The absolute constant in Property 4 we denote by $C$. Using this new notation, the bounded growth conditions are telling us that:

\begin{itemize}
\item For every $n$ and $I\in\II_n$ there is a set $C_I\subset\R^d$ s.t.

\begin{enumerate}
\item[(P1)] For each $m\ge n$, there are at most constant many $I'\in \II_m$ s.t. $B(C_I,r_m)\cap B(C_{I'},r_m)\neq\emp$.
\item[(P2)] $C_I$ contains a ball of radius $\sim r_n$.
\end{enumerate}
\item For every $m\ge n$ and $j\ge 0$ there are sets $E_{I,j}$ and their annuli $D_{I,j}=B(E_{I,j},r_m)\setminus E_{I,j}$, and $k_{n,m}\gtrsim r_n/r_m$, 
  satisfying the following: 
\begin{enumerate}
\item[(P3)] $C_I \subset E_{I,j}$ for every $I\in\II_n$ and for every $j$.
\item[(P4)] For each $m\ge n$ and $J \in \mathcal{I}_{m}$ there are at most $C$ many pairs $(I,j)$ with $I\in\II_n$, $j< k_{n,m}$ s.t. $D_{I,j}\cap C_{J}\neq \emp$.
\item[(P5)] For every $m>n$, $\sum_{I\in\II_n,j<k_{n,m}}\chi_{D_{I,j}(x)}\lesssim k_{n,m}r_m/r_n$ for every $x$.
\end{enumerate}
\item If $q\in\NN$ is large enough, then for every $n\ge k+q, m\ge n+q$ and $I\in\II_k$ there is an $\II_n'\subset\{I'\in \II_n\colon I'\subset I\}$, s.t. $\II_n'\neq \emp$, and the following hold:
  \begin{enumerate}
\item[(P6)] $D_{I',j}\subset C_I$ for every $I'\in \II_{n}'$, $j<k_{n,m}$.
\item[(P7)] If $D_{I',j}\cap C_{I''}\neq\emp$ for some $I'\in \II_{n}'$, $j<k_{n,m}$ and $I''\in\II_m$, then $I''\in\II_m'$.
\end{enumerate}
\end{itemize}

Proposition \ref{propo} below is a characterization for the existence of these sets. Note that it is a purely geometric statement about sets in $\R^d$; it has nothing to do with any modulus of continuity $\omega$ or distance function $\rho$. We do not assume in Proposition \ref{propo} that the sets $C_I, E_{I,j}$ were defined via any embedding of any space into $\R^d$.

For our characterization we need $k_{n,m}\gg 2^{(m-n)/d}$, more precisely, the assumptions we use are:
\begin{equation}\label{k1}
  \sup_n\sum_{m\ge n}2^{(m-n)/d}k_{n,m}^{-1}<\infty
  \end{equation}
  and
  \begin{equation}\label{k2}
    2^{(m-n)/d}k_{n,m}^{-1}<(4C)^{-1}\text{ if } m-n \text{ is large enough.}
    \end{equation}
    \begin{proposition}\label{propo} Let $C,q\in\NN$, let $\{r_n\}$ be a decreasing sequence converging to 0, and suppose that for every  $m\ge n$ we are given a $k_{n,m} \gtrsim r_n/r_m$ satisfying \eqref{k1} and \eqref{k2}. Then the following are equivalent:
      \begin{itemize}
      \item[(I)] There exist some sets for which (P1)-(P7) hold.
        \item[(II)] $\sum_n 2^{n/d}r_n <\infty$. 
\end{itemize}
\end{proposition}

\begin{remark}
  Note that (II) holds for a sequence $\{r_n\}$ if and only if it holds for $\{cr_n\}$ with a constant $c$. Also all other properties listed in this section remain valid when we replace each set $S\subset \R^d$ by $cS$ and we replace $r_n$ by $cr_n$. There wasn't any mathematical reason for replacing $cr_n$ by $r_n$ at the beginning of this section -- we did this only to simplify our formulas and notation (in this and in all subsequent sections).
\end{remark}

\begin{remark} We do not actually need the implication (II)$\implies$(I) in Proposition \ref{propo} for Theorem \ref{theorem} and Theorem \ref{chara}. We will get this for free, without any additional work, in Section \ref{construction}. Since it is an interesting statement, we included this for the sake of completeness.
\end{remark}

\subsection{Our answer to Problem \ref{main}}

We verified in Section \ref{bgrowth} that in the Sard and in the Cantor case part (I) of Proposition \ref{propo} holds with some $C,q\in\NN$. In the Sard case we had $\omega(r_n)=1/2^{n+1}$ (using our old definition of $r_n$), and $k_{n,m}\sim 2^{n-m}$. Thus we have \eqref{k1} and \eqref{k2} iff $d \geq 2$. In the Cantor case $\{r_n\}$ was an arbitrary sequence decreasing to 0, and we could choose $k_{n,m}$ to be arbitrary large.

It is an easy exercise to check that with the choice $\omega(r_n)=1/2^{n+1}$, (II) of Proposition \ref{propo} is equivalent to $\int_0^1\om(r)^{-1/d}\,dr<\infty$. Since (I) implies (II) in Proposition \ref{propo}, therefore indeed (ii) implies (i) in Theorem \ref{our}. Similarly, in the Cantor case, the implication (I)$\implies$(II) in Proposition \ref{propo} shows that if we have a positive answer to Problem \ref{main} then (II) holds.

In the general case, we say that a distance function $\rho$ has \emph{bounded growth}, if the following is true: if the answer to Problem \ref{main} is positive, i.e. the sets $A_I$ exist, then (I) of Proposition \ref{propo} holds (with the sequence $\{r_n\}$ defined by $\rho$ in Section \ref{seci} and with some appropriately chosen $C,q,k_{n,m}$ satisfying the assumptions in Proposition \ref{propo}). We know that in the Sard and in the Cantor example we always have bounded growth.

Now let $\rho$ be an arbitrary distance function and let $\{r_n\}$ be the sequence defined in Section \ref{seci}. Then it follows from (I)$\implies$(II) in Proposition \ref{propo} that:

\begin{theorem}\label{theorem} Suppose that the distance function $\rho$ has bounded growth. Then a positive answer to Problem \ref{main} implies $\sum_n 2^{n/d}r_n<\infty$.
\end{theorem}

For the reverse implication, we need to introduce one more notation. Let $r_n^*$ denote the supremum of $\rho(J_1,J_2)$, where the supremum is taken over all sequences $I_1, J_2\in\II$ whose length is at least $n$ and whose first $n-1$ terms agree and the $n^{th}$ term differs. That is, $n$ is the first index for which they are contained in two different intervals $I_1,I_2\in\II_n$. (In the Sard case $\omega(r_n^*)=1/2^{n-1}$ and in the Cantor case $r_n^*=r_n$. So in both cases $r_n^*\sim r_n$.)

For the following theorem we don't need to assume that $\rho$ has bounded growth. Our construction in Section \ref{construction} will show that:

\begin{theorem}\label{chara} If  $\sum_n 2^{n/d}r_n^*<\infty$, then there is a positive answer to Problem 4.
\end{theorem}

A corollary of this is that:

\begin{corollary} If $\rho$ has bounded growth, and $r_n\sim r_n^*$, then the answer to Problem 4 is positive if and only if $\sum_n 2^{n/d}r_n<\infty$.
\end{corollary}

In particular, this implies Theorem \ref{our}. 

\bigskip
In Section \ref{mainproof} we prove the implication (I)$\implies$(II) in Proposition \ref{propo}, and hence also Theorem \ref{theorem}. In Section \ref{construction} we will prove Theorem \ref{chara} and also the reverse implication in Proposition \ref{propo}.


\section{Proof of (I)$\implies$(II) in Proposition \ref{propo}}\label{mainproof}

\subsection{Beginning of the proof}

Suppose that (I) of Proposition \ref{propo} holds and (II) fails. Our aim is to find a contradiction.

Let $C$ be the absolute constant in (P\ref{4}). Choose $q\in\NN$ large enough so that (P\ref{6}) and (P\ref{7}) hold for $q$ and moreover \eqref{k2} holds if $m-n\ge q$.


If we split the divergent sum $\sum 2^{n/d}r_n$ into $q$ sums, at least one of them will be divergent. Without loss of generality we assume
\begin{equation}
  \label{*}
  \sum_n2^{qn/d}r_{qn}=\infty.
\end{equation}

\bigskip
Let $p_n^{d/(d-1)}$ denote the minimal measure among the $|C_I|$, where $I\in \II_n$.

\begin{remark} The motivation for our strange notation is that later in some sense $p_n$ will play the role of perimeter.
\end{remark}

To simplify our notation, denote
$$Q_n:=2^{qn/d}, \, R_n:=r_{qn},\, P_n:=p_{qn},\, K_{n,m}:=k_{qn,qm},\, S_{n,m}=\sum_{k=n+1}^m Q_kR_k.$$

By (P\ref{rn}), we have $R_n\lesssim P_n^{1/(d-1)}$ for every $n$. We will contradict \eqref{*}, i.e. contradict $\sum Q_nR_n=\infty$ by showing that:

\begin{lemma}\label{1}
  For every $n$ there is $m>n$ s.t. 
  \begin{equation}\label{11}
    S_{n,m}\lesssim Q_nP_n^{1/(d-1)}-Q_mP_m^{1/(d-1)}.
  \end{equation}
\end{lemma}
\begin{proof}[Finding a contradiction using Lemma \ref{1}] We take $n_1=1$ and using Lemma \ref{1} inductively choose a sequence $\{n_k\}_{k \in \mathbb{N}}$ such that $$ S_{n_k,n_{k+1}}\lesssim Q_{n_k}P_{n_{k+1}}^{1/(d-1)}-Q_{n_k}P_{n_{k+1}}^{1/(d-1)}.$$
Summing over all $k$, we have a telescopic sum on the right hand side. We have 
$$ \sum_{i=1}^{n_{k+1}}Q_iR_i \lesssim Q_1P_1^{1/(d-1)}.$$
Since the left hand side goes to infinity, we get a contradiction.
\end{proof}

Our aim is to prove Lemma \ref{1}. Fix $n$, and fix an $I\in\II_{qn}$ with $|C_I|=P_n^{d/(d-1)}$. Note that for $m=n+1$, the left hand side of $\eqref{11}$ is $Q_{n+1}R_{n+1}\le Q_{n+1}R_n\lesssim Q_nP_{n}^{1/(d-1)}$. Since also $Q_{n+1}\le Q_n$, therefore if $P_{n+1}\ll P_{n}$ then \eqref{11} holds with $m=n+1$. From now on, we assume that
\begin{equation}\label{I_0}
P_{n+1}\gtrsim P_{n}\quad\text{(for our fixed } n).
\end{equation}
Under this assumption, we will show a stronger statement than Lemma \ref{1}. Since $x^{d}-y^{d}\lesssim x^{d-1}(x-y)$ for every $x\ge y\ge 0$, we have
\begin{align*}
  |C_I\setminus \bigcup_{J\in\II_{qm}}C_J|&\le P_n^{d/(d-1)}-2^{q(m-n)}P_{m}^{d/(d-1)}\\
 &\lesssim P_n(P_n^{1/{(d-1)}}-2^{q(m-n)/d}P_m^{1/{(d-1)}})\\
&=P_nQ_n^{-1}(Q_nP_n^{1/(d-1)}-Q_mP_m^{1/(d-1)}).
\end{align*}

We will show that there is an $m>n$ s.t.
\begin{equation}\label{**}P_nQ_n^{-1}S_{n,m}\lesssim  |C_I\setminus \bigcup_{J\in\II_{qm}}C_J|.\end{equation}


\begin{remark}
  We will decompose the set in the right hand side of \eqref{**}, using the annuli described in the bounded growth sections. And then we will estimate its measure using a generalization of the isoperimetric inequality, which we describe in the next section.
\end{remark}
\subsection{Perimeters}

By the isoperimetric inequality,
$$|B(S,r)\setminus S|\gtrsim r|S|^{{(d-1)/d}}$$ for every set $S\subset\R^d$ and every $r>0$. A trivial corollary of this is that
$$|B(T,r)\setminus T|\gtrsim r|T|^{{(d-1)/d}}\ge  r|S|^{{(d-1)/d}}\quad\text{for every}\quad T\supset S.$$ As a refinement of this, we introduce in this section a sort of ‘minimal perimeter' of a bounded set $S$.

Let $S$ be an arbitrary bounded set, and let $u$ be the distance function from $S$. Then $u$ is Lipschitz with Lipschitz constant 1, hence it is differentiable almost everywhere, with $|\nabla u| \leq 1$. Moreover, it is easy to check that at every $x\not\in\cl(S)$, the directional derivative of $u$ is 1 along the line segment joining $x$ to its closest point $y\in\cl(S)$. Therefore $|\nabla u|=1$ at a.e. $x\not\in\cl(S)$, thus by the coarea formula we have
$$|B(S,r)\setminus S|\ge\int_{B(S,r)\setminus \cl(S)}|\nabla u|=\int_{0}^{r}\mathcal{H}^{d-1}(u^{-1}(t))\,dt=\int_{0}^{r}\HH^{d-1}(\partial B(S,t))\,dt.$$ Motivated by this, denote $$P_0(S):={\mathrm{ess}\,\inf}_{t>0}\HH^{d-1}(\partial B(S,t))$$ and $$P(S):=\inf_{T\supset S} P_0(T).$$
Then, clearly,
\begin{equation}\label{P0}
  P_0(S)=\inf_{r_1,r_2>0} |B(S,r_1+r_2)\setminus B(S,r_1)|/r_2,
\end{equation}
where, by the isoperimetric inequality,
$$|B(S,r_1+r_2)\setminus B(S,r_1)|/r_2\gtrsim |B(S,r_1)|^{{(d-1)/d}}\ge |S|^{{(d-1)/d}}.$$
Therefore $P_0(S)\gtrsim|S|^{{(d-1)/d}}$, and moreover, for every $T\supset S$, $P_0(T)\gtrsim|T|^{{(d-1)/d}}\ge |S|^{{(d-1)/d}}$. So
\begin{equation}
  \label{isop}
  P(S)\gtrsim |S|^{{(d-1)/d}}.
\end{equation}
By \eqref{P0}, for every $\eps>0$ there exist $T \supset S$ and $r_1, r_2>0$ such that
$$ |B(T,r_1+r_2)\setminus B(T,r_1)|/{r_2}\leq P_0(T)+\eps/2 \leq P(S)+\eps.$$
Taking $T'=B(T,r_1)$, we have that 
$ P(S) \leq |B(T',r_2)\setminus T'|/{r_2} \leq P(S)+\eps.$ Since $\eps$ is arbitrary, we have
\begin{equation}\label{tt}
   P(S)=\inf_{T\supset \cl(S),r>0} |B(T,r)\setminus T|/r,
 \end{equation}
 and for future reference let us note also its trivial corollary
  \begin{equation}\label{rr}
    |B(S,r)\setminus S|\ge  rP(S).
    \end{equation}
Further properties of $P(S)$ that we will be using are the following.

\begin{lemma}\label{peri} For every bounded sets $S_1,S_2\subset \mathbb{R}^d$ and $r >0$:
\begin{enumerate}
\item[(i)] $P(S_1 \cup S_2) \leq P(S_1) + P(S_2)$; 
\item[(ii)] $|B(S_1,r)\setminus (S_1 \cup S_2)|\ge rP(S_1)-rP(S_2)$.
  \end{enumerate}
\end{lemma}

Using induction, an immediate corollary is: for every bounded sets $S_j\subset \mathbb{R}^d$ and $r >0$:
\begin{equation}\label{key}
  |B(S_0,r)\setminus\bigcup_{j=0}^n S_n|\ge rP(S_0)-r\sum_{j=1}^nP(S_j).
\end{equation}

\begin{remark} The key properties we will need for the proof of our main theorem are \eqref{isop}, \eqref{rr} and \eqref{key}.
The reason for considering $P(S)$ instead of the simpler notions $|S|^{(d-1)/d}$ or $P_0(S)$ is that Lemma \ref{peri}, and hence also \eqref{key} do not hold with $P$ replaced by $|S|^{(d-1)/d}$ or $P_0$.
\end{remark}
  
\begin{proof}[Proof of Lemma \ref{peri}]
%
Property (i) follows easily from the observation that if $T_1\supset S_1$, $T_2\supset S_2$, and (say) $r_1\ge r_2$, then with $T:=B(T_1,r_1-r_2)\cup T_2$,
\begin{equation}\label{t01}
  \HH^{d-1}(\partial B(T,r_2))\le \HH^{d-1} (\partial B(T_1,r_1))+\HH^{d-1}(\partial B(T_2,r_2)).
  \end{equation} Indeed, for each $\eps>0$ there are positively many $r_1$ and positively many $r_2$ s.t.
\begin{equation}\label{t12}
\HH^{d-1}(\partial B(T_1,r_1))\le P_0(T_1)+\eps,\quad\HH^{d-1}(\partial B(T_2,r_2))\le P_0(T_2)+\eps,
\end{equation}
and without loss of generality we can assume that there are positively many $r_1,r_2$ satisfying \eqref{t12} for which $r_1\ge r_2$. Then we can fix an $r>0$ for which there are positively many $r_1,r_2$ satisfying  \eqref{t12} with $r=r_1-r_2$. With $T=B(T_1,r)\cup T_2$, from \eqref{t01} and \eqref{t12} we get $P(S_1\cup S_2)\le P_0(T)\le P_0(T_1)+ P_0(T_2)+2\eps$. Since this holds for every $T_1\supset S_1$, $T_2\supset S_2$ and $\eps>0$, therefore (i) follows.

In order to prove (ii), again fix an $\eps>0$. Then by \eqref{tt} and by the co-area formula there is a $T'\supset S_2$ and $r>0$ s.t. $$\frac{1}{r}\int_0^r\HH^{d-1}(\partial B(T',t))\,dt=\frac{1}{r}|B(T',r)\setminus \cl(T')|<P(S_2)+\eps/2.$$
Let $t'\in(0,r)$ to be a Lebesgue density point of $t\to \HH^{d-1}(\partial B(T',t))$ for which $\HH^{d-1}(\partial B(T',t'))<\eps/2+{\mathrm{ess}}\inf_{t\in(0,r)}\HH^{d-1}(\partial B(T',t))$. Then with $T_2:=B(T',t')$  and an $r_{\eps}$ small enough we have 
$$ \frac{1}{r}\int_0^r\HH^{d-1}(\partial B(T_2,t))\,dt<P(S_2)+\eps \text{ for all }r<r_{\eps}.$$
Then for any $T \supset S_1$ and $r<r_{\eps}$, we have
\begin{align*}
P(T)\leq P(T\cup T_2) & \leq \frac{1}{r}\int_{0}^{r}\HH^{d-1}(\partial B(T \cup T_2,t))\,dt \\
    & \leq \frac{1}{r}\left(\int_{0}^{r} \HH^{d-1}(\partial B(T,t)\setminus T_2)\,dt
   + \int_{0}^{r} \HH^{d-1}(\partial B(T_2,t))\,dt\right)\\
   & \le \frac{1}{r}\int_{0}^{r} \HH^{d-1}(\partial B(T,t)\setminus T_2)\,dt + P(S_2)+\eps.
\end{align*}
Since $T_2\supset \cl(S_2)$ we have
\begin{align*}
P(S_1)-P(S_2)-\eps&\le P(T)-P(S_2)-\eps \leq \frac{1}{r}\int_{0}^{r} \HH^{d-1}(\partial B(T,t)\setminus T_2)\,dt\\
  &\le \frac{1}{r}\int_{0}^{r} \HH^{d-1}(\partial B(T,t)\setminus \cl(S_2))\,dt
\end{align*}
for all $r< r_{\eps}$.
This holds for every $T \supset S_1$, in particular it holds for $T=B(S_1,t_0)$ with every $t_0> 0$. Taking infimum for all $t_0>0$ and all $r<r_\eps$ we can see that $$ P(S_1)-P(S_2)-\eps\le  {{\mathrm{ess}}\inf}_{t>0}\HH^{d-1}(\partial B(S_1,t)\setminus \cl(S_2)).$$

Applying the co-area formula for the distance function from $S_1$, on the domain $B(S_1,r)\setminus (\cl(S_1)\cup \cl(S_2))$, we obtain 
\begin{align*}
    r(P(S_1)-P(S_2)-\eps) & \leq \int_{0}^{r}\HH^{d-1}(\partial B(S_1,t)\setminus \cl(S_2))\,dt \\
    & = |B(S_1,r)\setminus(\cl(S_1)\cup \cl (S_2))|\le  |B(S_1,r)\setminus(S_1\cup S_2)|.
\end{align*}
Since this holds for every $\eps>0$, the proof is finished.
\end{proof}

\subsection{Estimating the right hand side of \eqref{**}}

Denote $$T_m:=R_m\sum_{J\in\II_{qm}'}P(C_J).$$ By \eqref{rr} this is a lower estimate for $\sum|B(C_J,R_m)\setminus C_J|$, and by (P\ref{CI}) the sets $B(C_J,R_m)$ have bounded overlap.

Therefore $T_m\lesssim 1$ for every $m$. On the other hand, for $n+1$ we have from \eqref{I_0} the lower estimate
\begin{equation*}
  T_{n+1}\gtrsim R_{n+1}\sum_{J\in\II_{q(n+1)}'}|C_J|^{(d-1)/d}\ge R_{n+1}P_{n+1}\gtrsim R_{n+1}P_{n}\gtrsim P_nQ_n^{-1}S_{n,n+1}.
\end{equation*}
So \eqref{**} would hold, with $m=n+1$, if we could replace its right hand side by $T_{n+1}$.

\begin{remark} In the Cantor example we can indeed replace it by $T_{n+1}$. Which makes everything much easier. However, in the Sard, and in the general case, we cannot, we will prove only a (somewhat) weaker lower bound, with an error term. See \eqref{3*} below. In the next section our goal will be to make the error term small enough (compared to the main term). 
\end{remark}

  Note that in all the above formulas we sum only for $\II_{qm}'$ and not for  $\II_{qm}$. We apply in this section the results of Section \ref{section5} for the interval $I$ we fixed. 

%
%
  

By (P\ref{5}) and (P\ref{6}), for each $M>N>n$ we have
$$\frac{r_{M}}{r_N}\cdot K_{N,M}|C_{I}\setminus\bigcup_{J\in \II_M} C_{J}|\gtrsim\sum_{\substack{I_0\in\II_N'\\ 0\le j< K_{N,M}}}|D_{I_0,j}\setminus\bigcup_{J\in \II_M} C_{J}|,$$
and by \eqref{key}
$$|D_{I_0,j}\setminus\bigcup_{J\in \II_M} C_{J}|\ge r_{M}P(E_{I_0,j})-r_{M}\sum_{\substack{J\in \II_M\\ C_{J}\cap D_{I_0,j}\neq\emp}}P(C_{J}).$$
By (P\ref{3}), we have $P(E_{I_0,j})\ge P(C_{I_0})$, and by (P\ref{4}) and (P\ref{7}),
$$\sum_{\substack{I_0\in\II_N\\0\le j< K_{N,M}}}\sum_{\substack{J\in \II_M\\ C_J\cap D_{I_0,j}\neq\emp}}P(C_J)\le C\sum_{\substack{J\in \II_M\\ \exists (I_0,j)\colon C_J\cap D_{I_0,j}\neq\emp}}P(C_{J})\le C\sum_{J\in\II_M'}P(C_{J}).$$
where $C$ is the absolute constant in (P\ref{4}).

Putting these estimates together we get
\begin{align*}
  |C_{I}\setminus\bigcup_{J\in \II_M} C_{J}|&\gtrsim r_N(\sum_{J\in\II_N'}P(C_J)-CK_{N,M}^{-1}\sum_{J\in\II_M'}P(C_J))\\
  &=T_N-C(K_{N,M}r_M/r_N)^{-1}T_M.
\end{align*}

Recall that our aim is to show that \eqref{**} holds for some (carefully chosen) $m$. We will prove this by showing that there is $n<N<M$ s.t.
\begin{equation}\label{3*}
  P_nQ_n^{-1}S_{n,M}\lesssim T_N-C(K_{N,M}r_M/r_N)^{-1}T_M.
  \end{equation}
We have already seen that
\begin{equation}\label{4*}P_nQ_n^{-1}S_{n,n+1}\lesssim T_{n+1}.
  \end{equation}

  \subsection{Finding $N$ and $M$}\label{sub}

To further simplify notation, for every $m\ge n$ denote
$$a_{m}:=P_nQ_n^{-1}Q_mR_m.$$ Using this notation, $a_{m+1}\lesssim a_m$ for every $m$, $\sum_m a_m=\infty$, 
and our aim is to show that
\begin{equation}\label{maine}
  \sum_{k=n+1}^M a_k\lesssim T_N-CK_{N,M}^{-1}2^{q(M-N)/d}T_Ma_N/a_M
  \end{equation}
for some $n<N<M$.
For simplicity, also denote $$L_{N,M}:=2CK_{N,M}^{-1}2^{q(M-N)/d}$$ for every $n<N<M$. We know from \eqref{k1} and \eqref{k2} that  $\sum_M L_{N,M}\lesssim 1$ for every $N$, and  $L_{N,M}\le 1/2$ for every $N,M$.

We have $\limsup_{M\to\infty} L_{N,M}^{-1}a_M=\infty$ (otherwise $\infty=\sum_M a_M\lesssim \sum_M L_{N,M}<\infty$, which is a contradiction). Therefore, inductively, we can choose a subsequence, starting from $n_0=n+1$, and then by choosing, after each $a_{n_k}$, the first $a_{n_{k+1}}$ for which
$$a_{n_{k+1}}>L_{n_{k},n_{k+1}}a_{n_k}.$$ 

Since $a_l\le L_{n_{k},l}a_{n_k}< a_{n_k}$ for $n_k< l<n_{k+1}$, and $a_{n_{k+1}}\lesssim a_{n_{{k+1}}-1}$, therefore
\begin{equation}\label{ineq}a_{n_{k+1}}\lesssim a_{n_k}.
  \end{equation}
  Also
$$\sum_{l=n_k}^{n_{k+1}-1}a_l=a_{n_k}+\sum_{l=n_k+1}^{n_{k+1}-1}a_l\le a_{n_k}+(\sum_{l=n_k+1}^{n_{k+1}-1}L_{n_k,l})a_{n_k}\sim a_{n_k},$$ so
  \begin{equation}\label{ineq2}
    \sum_{l=n_k}^{n_{k+1}-1}a_l\sim a_{n_k}.
    \end{equation}
    Inequalities \eqref{ineq} and \eqref{ineq2}, together with
    \begin{align*}
      a_{n_m}/a_{n_{k-1}}&=(a_{n_m}/a_{n_{m+1}})\cdot(a_{n_{m+1}}/a_{n_{m+2}})\cdots(a_{n_{k-2}}/a_{n_{k-1}})\\      
&\lesssim (L_{n_m,n_{m+1}}L_{n_{m+1}1,n_{m+2}}\dots L_{n_{k-2},n_{k-1}} )^{-1}\\
                           \end{align*} give
      \begin{align*}\sum_{l=n_0}^{n_{k}}a_l&\lesssim\sum_{m=0}^{k} a_{n_m}\lesssim\sum_{m=0}^{k-1} a_{n_m}= a_{n_{k-1}}\sum_{m=0}^{k-1} a_{n_m}/a_{n_{k-1}}\\
 &\lesssim a_{n_{k-1}}\sum_{m=0}^{k-1}(L_{n_m,n_{m+1}}L_{n_{m+1},n_{m+2}}\dots L_{n_{k-2},n_{k-1}} )^{-1}.
\end{align*}
By our assumption $L_{N,M}\le 1/2$, in the last sum each term is at most half of the previous term. Therefore we can estimate the sum by its first term, and we get
$$\sum_{l=n_0}^{n_{k}}a_l\lesssim a_{n_{k-1}}(L_{n_0,n_1}L_{n_1,n_2}\dots L_{n_{k-2},n_{k-1}} )^{-1}:=s_{n_{k-1}}.$$

Since the left hand side of this converges to infinity, therefore $s_{n_k}\to\infty$ as $k\to\infty$.



We defined $s_{n_{m-1}}$ as an upper bound of the left hand side of $\eqref{maine}$ (when we replace $M$ by $n_m$ in $\eqref{maine}$). We will prove $\eqref{maine}$ by showing that there is an $m\ge 1$ s.t. with $N=n_{m-1}$, $M=n_m$,  
\begin{equation}\label{12}
  s_{N}\lesssim T_N-2^{-1}L_{N,M}T_Ma_N/a_M.
  \end{equation}
  Using \eqref{4*} we have $s_{n_0}=a_{n_0}\lesssim T_{n_0}$. The fact that $s_{n_k}$ goes to infinity and $T_{n_k}$ is bounded allows us to choose $m$ to be the first index for which this inequality fails. That is, with some absolute constant $c_0$ and some $m>0$, $N=n_{m-1}$, $M=n_m$ we have $s_N\le c_0 T_N$ and $s_M>c_0T_M$. In particular, $T_M/T_N<s_M/s_N=L_{N,M}^{-1}a_M/a_N$, i.e. $L_{N,M}T_Ma_N/a_M<T_N$. Using this, the right hand side of \eqref{12} is
  $$T_N-2^{-1}L_{N,M}T_Ma_N/a_M>T_N/2.$$
  This together with $s_N\le c_0 T_N$ finishes the proof of \eqref{12} and hence finishes the proof of Lemma \ref{1}.

  \section{Construction}\label{construction}
  Let $\{v_n\}$ be an arbitrary sequence decreasing to 0, and assume that $\sum_n2^{n/d}v_n<\infty$. We split it into $d$ subsequences $\{v_{dn+m}\}_{n=1}^\infty$, $m=0,1,\dots,d-1$. For each $m$, take a Cantor set lying on the $m^{th}$ coordinate axes of $\R^d$, whose gaps have length $v_{dn+m}$ in the $n^{th}$ stage of the standard Cantor set construction (since it has $2^{n-1}$ gaps of length $v_{dn+m}$, and $\sum_n 2^{n-1}v_{dn+m}<\infty$, therefore such a Cantor set exists).

  Now consider a finite sequence of $0's$ and $1's$, say, $I=x_1x_2\dots x_N$. We can split this into $d$ subsequences, by putting into the $m^{th}$ subsequence those $x_j$ for which $j\equiv m\,(\textrm{mod } d)$. For the $m^{th}$ subsequence, let $I_m$ denote the corresponding construction interval of the Cantor set on the $m^{th}$ coordinate axis. And let $R_I=R_{x_1x_2\dots x_N}$ denote the rectangle $R_I=I_0\times I_1\times\dots\times I_{m-1}$. It is immediate to see that if we have another sequence $I'=x_1'x_2'\dots x_N'$, and $x_k\neq x_k'$ for some $1\le k\le N$, then the distance of $R_I$ and $R_{I'}$ is at least $v_k$: indeed, if $k\equiv m\,(\textrm{mod }d)$, then the sets $R_I,R_{I'}$ project onto two different intervals in the $m^{th}$ direction and these two intervals are separated by a gap of length at least $v_k$.
  
  \begin{proof}[Proof of Theorem \ref{chara}] Put $v_n:=r_n^*$ and $A_I:=R_I$ in the above construction.
  \end{proof}

  \begin{proof}[Proof of (II)$\implies$(I) in Proposition \ref{propo}] Let $\{r_n\}$ be the sequence in (II) of Proposition \ref{propo}. Consider the Cantor case, with this sequence $\{r_n\}$. By Theorem \ref{chara} in this case we have a positive answer to Problem \ref{main}. And since in the Cantor case we have bounded growth, therefore we know that there are sets satisfying (P1)-(P7) with some $C,q,k_{n,m}$. Moreover, in the Cantor case we can choose $C$ to be anything we like (since in (P4), instead of 'constant many pairs $(I,j)$', there are none), we can also choose $q$ to be anything we like (since we have seen that (P6)-(P7) hold with $q=1$), and we have also seen that we can choose $k_{n,m}$ to be arbitrary. This shows that indeed for any given $C,q,k_{n,m}$ there are sets satisfying (P1)-(P7), i.e. (I) of Proposition \ref{propo} holds.
\end{proof}
\bibliographystyle{plain}
\bibliography{bibliography}
\Addresses
\end{document}